\documentclass[leqno,12pt]{amsart}

\usepackage{amscd}
\usepackage{amsmath}
\usepackage{amssymb}
\usepackage{amsthm}
\usepackage{appendix}
\usepackage{bbm}
\usepackage{enumerate}
\usepackage{enumitem}
\usepackage{eucal}
\usepackage{geometry}
\usepackage{graphicx}
\usepackage{mathtools}
\usepackage{multirow}
\usepackage{setspace}
\usepackage{url}
\usepackage[T1]{fontenc}
\usepackage[table]{xcolor}

\usepackage{float}
\usepackage{microtype}
\usepackage{verbatim}

\usepackage{tikz}
\usepackage{tikz-cd}
\usetikzlibrary{matrix,arrows,decorations.pathmorphing}

\usepackage[all,cmtip]{xy}

\usepackage[colorlinks,linkcolor=blue,anchorcolor=blue,citecolor=green,backref=page]{hyperref}
\usepackage[alphabetic]{amsrefs}

\newtheorem{theorem}{Theorem}[section]

\newtheorem{conjecture}[theorem]{Conjecture}
\newtheorem{corollary}[theorem]{Corollary}
\newtheorem{lemma}[theorem]{Lemma}
\newtheorem{proposition}[theorem]{Proposition}
\newtheorem{question}[theorem]{Question}

\theoremstyle{definition}

\newtheorem{definition}[theorem]{Definition}

\newtheorem{remark}[theorem]{Remark}

\newtheorem{setup}[theorem]{}

\newtheorem*{ack}{Acknowledgements}
\newtheorem*{claim*}{Claim}
\newtheorem*{note-term}{Notation and Terminology}

\numberwithin{equation}{section}

\newcommand{\RomanNumeralCaps}[1]{\MakeUppercase{\romannumeral #1}}

\newcommand{\id}{\operatorname{id}}

\newcommand{\Spec}{\operatorname{Spec}}

\newcommand{\bk}{\mathbf{k}}

\newcommand{\bP}{\mathbb{P}}
\newcommand{\bQ}{\mathbb{Q}}

\newcommand{\bR}{\mathbb{R}}

\newcommand{\cO}{\mathcal{O}}

\newcommand{\oK}{\overline{K}}

\title[Potential density]
{Potential density of projective varieties having an int-amplified endomorphism}
\author{Jia Jia, Takahiro Shibata, and De-Qi Zhang}

\address{
\textsc{National University of Singapore, Singapore 119076, Republic of Singapore}
}
\email{jia\_jia@u.nus.edu}

\address{
\textsc{National University of Singapore, Singapore 119076, Republic of Singapore}
}
\email{mattash@nus.edu.sg}

\address{
\textsc{National University of Singapore, Singapore 119076, Republic of Singapore}
}
\email{matzdq@nus.edu.sg}

\subjclass[2020]{
37P55, 
14G05, 
14E30, 
08A35. 
}
\keywords{Potential density, Int-amplified endomorphism, Arithmetic degree, Dynamical degree}

\begin{document}

\begin{abstract}
We consider the potential density of rational points on an algebraic variety defined over a number field $K$, i.e., the property that the set of rational points of $X$ becomes Zariski dense after a finite field extension of $K$.
For a non-uniruled projective variety with an int-amplified endomorphism, we show that it always satisfies potential density.
When a rationally connected variety admits an int-amplified endomorphism, we prove that there exists some rational curve with a Zariski dense forward orbit, assuming the Zariski dense orbit conjecture in lower dimensions.
As an application, we prove the potential density for projective varieties with int-amplified endomorphisms in dimension $\leq 3$.
We also study the existence of densely many rational points with the maximal arithmetic degree over a sufficiently large number field.
\end{abstract}

\maketitle
\setcounter{tocdepth}{1}
\tableofcontents

\section{Introduction}\label{s:intr}

Let $K$ be a number field with a fixed algebraic closure $\oK$.
Given a variety $X$ over $K$, we are interested in the set of $K$-rational points $X(K)$ of $X$.
More specifically, we study the \textit{potential density} of varieties over $K$.

\begin{definition}
A variety $X$ defined over a number field $K$ is said to satisfy \textit{potential density} if there is a finite field extension $K \subseteq L$ such that
$X_L(L)$ is Zariski dense in $X_L$, where $X_L \coloneqq X \times_{\Spec K} \Spec L$.
\end{definition}

The potential density of varieties over number fields has been investigated in several papers.
The potential density problem is attractive because the potential density of a variety is pretty much governed by its geometry.
See \cite{campana2004orbifolds} for a conjecture characterising varieties satisfying potential density.
However, algebraic varieties for which the potential density is verified are very few.
See \cite{hassett2003potential} for a survey of studies on the potential density problem.

In this paper, we first study the potential density of varieties admitting int-amplified endomorphisms.
For the definition of int-amplified endomorphisms, see \ref{notation}\ref{n:int_amp}.
Recently, the equivariant minimal model program for varieties with int-amplified endomorphisms was established (cf.~\cite{meng2020semi}).
It has been used to study arithmetic-dynamical problems (cf.~\cite{matsuzawa2019kawaguchi}, \cite{matsuzawa2020nondensity}).
It turns out that the equivariant minimal model program is also useful for the potential density problem.

Our main conjecture is the following.

\begin{conjecture}[Potential density under int-amplified endomorphisms]
\label{Conj:PD_int_A}
Let $X$ be a projective variety defined over a number field $K$.
Suppose that $X$ admits an int-amplified endomorphism.
Then $X$ satisfies potential density.
\end{conjecture}

The endomorphism being int-amplified is a crucial assumption in Conjecture~\ref{Conj:PD_int_A} above.
Indeed, consider $X=X_1\times C$ where $X_1$ is any smooth projective variety and $C$ is any smooth projective curve of genus at least $2$.
Such $X$ does not satisfy potential density (cf.~Remark~\ref{r:after_conj}\ref{r:nonPD}).
It does not have any int-amplified endomorphisms either;
this is because every surjective endomorphism $f$ of $X$, after iteration, has the form $(x_1,x_2)\mapsto(g(x_1,x_2),x_2)$ for some morphism $g\colon X_1\times C\to X_1$ by \cite{sano2020dynamical}*{Lemma 4.5}, and hence descends to the identity map $\id_C$ on $C$ via the natural projection $X \to C$;
thus, the iteration and hence $f$ itself are not int-amplified (cf.~\cite{meng2020building}*{Lemma 3.7 and Theorem 1.1}).

One might think that Conjecture~\ref{Conj:PD_int_A} is too strong.
In fact, the following even stronger conjecture has already been long outstanding (cf.~Medvedev--Scanlon~\cite{medvedev2009polynomial}*{Conjecture 5.10}, and Amerik--Bogomolov--Rovinsky~\cite{amerik2011remarks}).

\begin{conjecture}[Zariski dense orbit conjecture]
\label{Conj:ZDO}
Let $X$ be a variety defined over an algebraically closed field $\bk$ of characteristic zero and $f\colon X\dashrightarrow X$ a dominant rational map.
If the $f^*$-invariant function field $\bk(X)^f$ is trivial, that is, $\bk(X)^f=\bk$, then there exists some $x\in X(\bk)$ whose (forward) $f$-orbit $O_f(x)\coloneqq\{f^n(x)\mid n\geq 0\}$ is well-defined (i.e., $f$ is defined at $f^n(x)$ for any $n\geq 0$) and Zariski dense in $X$.
\end{conjecture}

Note that Conjecture~\ref{Conj:ZDO} with $f$ being int-amplified implies Conjecture~\ref{Conj:PD_int_A} (cf.~Lemmas~\ref{l:int-amp_ZDO} and \ref{l:subvar}).

\begin{remark}\label{r:after_conj}
We recall some known cases of the potential density problem and Conjecture~\ref{Conj:ZDO}.
\begin{enumerate}[leftmargin=2em, ref=(\arabic*)]
	\item \label{r:ur_ab_PD}
	Unirational varieties and abelian varieties over number fields satisfy potential density (cf.~\cite{hassett2003potential}*{Corollary 3.3 and Proposition 4.2}).
	\item \label{r:nonPD} Let $X$ be a variety with a dominant rational map $X\dashrightarrow C$ to a curve of genus $\geq 2$ over a number field.
	Then $X$ does not satisfy potential density (cf.~\cite{faltings1983endlichkeitssatze} and \cite{hassett2003potential}*{Proposition 3.1}).
	\item \label{r:ZDO_curve}
	Conjecture~\ref{Conj:ZDO} holds for any pair $(X, f)$ with $X$ being a curve (cf.~\cite{amerik2011existence}*{Corollary 9}).
	\item \label{r:ZDO_surf}
	Conjecture~\ref{Conj:ZDO} holds for any pair $(X, f)$ with $X$ being a projective surface and $f$ a surjective endomorphism of $X$ (cf.~\cite{xie2019existence}, \cite{jia2020surjective}).
\end{enumerate}
\end{remark}

We first prove Conjecture~\ref{Conj:PD_int_A} for rationally connected varieties in dimension $\leq 3$.

\begin{proposition}\label{t:pd_rc}
Let $X$ be a rationally connected projective variety over $K$.
Suppose that $\dim X \leq 3$ and $X$ admits an int-amplified endomorphism.
Then $X$ satisfies potential density.
\end{proposition}

Conjecture~\ref{Conj:PD_int_A} also has a positive answer for non-uniruled varieties in any dimension:

\begin{proposition}\label{t:non-ur}
Let $X$ be a non-uniruled projective variety over $K$.
Suppose that $X$ admits an int-amplified endomorphism.
Then $X$ satisfies potential density.
\end{proposition}

With the help of Propositions~\ref{t:pd_rc} and~\ref{t:non-ur}, we are able to show:

\begin{theorem}\label{t:main}
Let $X$ be a normal projective variety over $K$ with at worst $\bQ$-factorial klt singularities.
Suppose that $\dim X\leq 3$ and $X$ admits an int-amplified endomorphism.
Then $X$ satisfies potential density.
\end{theorem}

In the last section, we study Question~\ref{q:dr} below, which is also arithmetic in nature, initiated in \cite{kawaguchi2014examples} and further studied in \cite{sano2020zariski} and \cite{sano2021zariski}.

\begin{definition}[cf.~\cite{sano2020zariski}*{Definition 1.4}]
Let $X$ be a projective variety over a number field $K$ and $f\colon X\to X$ a surjective morphism.
We recall the inequality
\[
	\alpha_f(x) \leq d_1(f)
\]
between the arithmetic degree $\alpha_f(x)$ at a point $x \in X(\oK)$ and the first dynamical degree $d_1(f)$ of $f$ (cf.~\ref{notation}\ref{n:d_deg} and \ref{n:a_deg}).
Let $L$ be an intermediate field: $K \subseteq L \subseteq \oK$.
We say that $(X, f)$ has \textit{densely many $L$-rational points with the maximal arithmetic degree} if there is a subset $S\subseteq X(L)$ satisfying the following conditions:
\begin{enumerate}
	\item $S$ is Zariski dense in $X_L$;
	\item the equality $\alpha_f(x)=d_1(f)$ holds for all $x\in S$; and
	\item $O_f(x_1)\cap O_f(x_2)=\emptyset$ for any pair of distinct points $x_1,x_2\in S$.
\end{enumerate}

Following \cite{sano2021zariski}, we introduce the following notation.
We say that $(X,f)$ satisfies $(DR)_L$ if ($X, f)$ has densely many $L$-rational points with the maximal arithmetic degree.
We say that $(X,f)$ satisfies $(DR)$ if there is a finite field extension $K \subseteq L$ ($\subseteq \oK$) such that $(X, f)$ satisfies $(DR)_L$.
\end{definition}

\begin{question}\label{q:dr}
Let $X$ be a projective variety over $K$ and $f\colon X\to X$ a surjective endomorphism.
Assume that $X$ satisfies potential density.
Does $(X, f)$ satisfy $(DR)$?
\end{question}

Question~\ref{q:dr} has a positive answer for smooth projective surfaces when $d_1(f) > 1$ (cf.~\cite{sano2021zariski}*{Theorem 1.5}).
We generalise it to (possibly singular) projective surfaces:

\begin{theorem}\label{t:surf}
Let $X$ be a normal projective surface over $K$ satisfying potential density, and $f \colon X \to X$ a surjective morphism with $d_1(f) > 1$.
Then $(X, f)$ satisfies $(DR)$.
\end{theorem}

The following is an affirmative answer to Question~\ref{q:dr} for int-amplified endomorphisms on rationally connected threefolds.

\begin{theorem}\label{t:rc_dr}
Let $X$ be a rationally connected smooth projective threefold over $K$ and $f\colon X\to X$ an int-amplified endomorphism.
Then $(X, f)$ satisfies $(DR)$.
\end{theorem}

\begin{ack}
The first, second and third authors are supported, from NUS, by the President's scholarship, a Research Fellowship and an ARF, respectively.
\end{ack}

\section{Preliminaries}\label{s:preliminary}

\begin{setup}\label{notation}
\textbf{Notation and Terminology}

\begin{enumerate}[leftmargin=2em, ref=(\arabic*)]
	\item Let $K$ be a number field.
	We work over $K$ when considering the potential density.
	We fix an algebraic closure $\oK$ of $K$.
	\item Let $\bk$ be an algebraically closed field of characteristic zero.
	We work over $\bk$ when considering geometric properties.
	\item A \textit{variety} means a geometrically integral separated scheme of finite type over a field.
	\item Let $X$ be a variety over $K$ and $f\colon X\to X$ a morphism (over $K$).
	We denote $X_{\oK}\coloneqq X\times_{\Spec \oK} \Spec K$ and $f_{\oK}\colon X_{\oK}\to X_{\oK}$ the induced morphism (over $\oK$).
	\item The symbol $\sim_{\bR}$ denotes the $\bR$-linear equivalence on Cartier divisors.
	\item We refer to \cite{kollar1998birational} for definitions of $\bQ$-factoriality and klt singularities.
	\item \label{n:ur}
	A variety $X$ of dimension $n$ is \textit{uniruled} if there is a variety $U$ of dimension $n-1$ and a dominant rational map $\bP^1\times U\dashrightarrow X$.
	\item \label{n:rc}
	Let $X$ be a proper variety over a field $k$.
	We say that $X$ is \textit{rationally connected} if there is a family of proper algebraic curves $U\to Y$ whose geometric fibres are irreducible rational curves with cycle morphism $U\to X$ such that $U\times_Y U\to X\times X$ is dominant (cf.~\cite{kollar1996rational}*{\RomanNumeralCaps{4} Definition 3.2}).
	When $k$ is algebraically closed of characteristic zero, if $X$ is rationally connected, then any two closed points of $X$ are connected by an irreducible rational curve over $k$ (by applying \cite{kollar1996rational}*{\RomanNumeralCaps{4} Theorem 3.9} to a resolution of $X$).
	The converse holds when $k$ is also uncountable (cf.~\cite{kollar1996rational}*{\RomanNumeralCaps{4} Proposition 3.6.2}).
	\item A normal projective variety $X$ is said to be \textit{\textit{Q}-abelian} if there is a finite surjective morphism $\pi\colon A \to X$, which is {\'e}tale in codimension $1$, with $A$ being an abelian variety.
	\item For a morphism $f \colon X \to X$ and a point $x \in X$, the forward \textit{$f$-orbit} of $x$ is the set $O_f(x)\coloneqq\{x, f(x), f^2(x), \ldots\}$.
	We denote the Zariski closure of $O_f(x)$ by $Z_f(x)$.

	More generally, for a closed subset $Y \subseteq X$, we denote $O_f(Y)\coloneqq\bigcup_{n=0}^\infty f^n(Y)$ and its Zariski-closure $Z_f(Y)\coloneqq\overline{O_f(Y)}$.
	We say that $O_f(Y)$ is Zariski dense if $Z_f(Y)=X$.
	\item \label{n:int_amp}
	A surjective morphism $f \colon X \to X$ of a projective variety is called \textit{int-amplified} if there exists an ample Cartier divisor $H$ on $X$ such that $f^*H - H$ is ample.
	In particular, polarised endomorphisms are int-amplified.
	\item \label{n:d_deg}
	Let $X$ be a projective variety and $f \colon X \to X$ a surjective morphism.
	The \textit{first dynamical degree} $d_1(f)$ of $f$ is the limit
	\[
		d_1(f)\coloneqq\lim_{n \to \infty} ((f^n)^*H \cdot H^{\dim X-1})^{1/n},
	\]
	where $H$ is an ample Cartier divisor on $X$.
	This limit always converges and is independent of the choice of $H$.
	\item \label{n:a_deg}
	Let $X$ be a projective variety over $K$ and $f \colon X \to X$ a surjective morphism.
	Fix a (logarithmic) height function $h_H \geq 1$ associated to an ample Cartier divisor $H$ on $X$.
	For $x \in X(\overline K)$, the \textit{arithmetic degree} $\alpha_f(x)$ of $f$ at $x$ is the limit
	\[
		\alpha_f(x)\coloneqq\lim_{n \to \infty} h_H(f^n(x))^{1/n}.
	\]
	This limit always converges and is independent of the choices of $H$ and $h_H$ (cf.~\cite{kawaguchi2016dynamical}).
\end{enumerate}
\end{setup}

\begin{lemma}\label{l:int-amp_ZDO}
Let $X$ be a projective variety over $\bk$ and $f\colon X\to X$ an int-amplified endomorphism.
Then $\bk(X)^f=\bk$.
In particular, if Conjecture~\ref{Conj:ZDO} holds for $(X, f)$, then there exists some $x\in X(\bk)$ such that $O_f(x)$ is Zariski dense in $X$.
\end{lemma}

\begin{proof}
Assume to the contrary that there is a nonconstant rational function $\phi\colon X\dashrightarrow \bP^1$ such that $\phi\circ f=\phi$.
Let $\Gamma$ be the graph of the rational map $\phi\colon X\dashrightarrow \bP^1$ with projections $\pi_1\colon\Gamma\to X$ being birational and $\pi_2\colon\Gamma\to\bP^1$ being surjective.
Then $f$ lifts to an endomorphism $f|_{\Gamma}$ on $\Gamma$ such that $\pi_1\circ f|_{\Gamma}=f\circ\pi_1$ and $\pi_2\circ f|_{\Gamma}=\pi_2$.
It follows from \cite{meng2020building}*{Lemmas 3.4 and 3.5} that $\id \colon \bP^1 \to \bP^1$ is int-amplified, which is absurd.
\end{proof}

\begin{lemma}\label{l:subvar}
Let $X$ be a projective variety over $K$, $f \colon X \to X$ a surjective morphism, and $Z \subseteq X$ a subvariety which satisfies potential density (e.g., $Z$ is an abelian variety or unirational;
see Remark~\ref{r:after_conj}\ref{r:ur_ab_PD}).
If $O_f(Z)$ is Zariski dense, then $X$ satisfies potential density.
\end{lemma}

\begin{proof}
Replacing $K$ with a finite extension, we may assume that $Z(K)$ is Zariski dense in $Z$.
Then the union $\bigcup_{n=0}^\infty f^n(Z(K))$ is a Zariski dense set of $K$-rational points of $X$.
\end{proof}

\section{Rationally connected varieties: Proof of Proposition~\ref{t:pd_rc}}
\label{s:rc}

\begin{lemma}\label{l:pd_rc}
Let $X$ be a rationally connected projective variety over $\bk$ and of dimension $d \geq 1$, and $f\colon X\to X$ an int-amplified endomorphism.
Assume Conjecture~\ref{Conj:ZDO} in dimension $\leq d-1$.
Then there is a rational curve $C \subseteq X$ such that $O_f(C)$ is Zariski dense.
\end{lemma}

\begin{proof}
If we have a Zariski dense $f$-orbit $O_f(x)$, take any rational curve $C$ passing through $x$.
Clearly, $O_f(C)$ is Zariski dense.
So we may assume that $f$ has no Zariski dense orbit.

Let $x\in X(\bk)$ be a point such that $Z_f(x)$ is irreducible with dimension $r < d$.
By \cite{fakhruddin2003questions}*{Theorem 5.1}, the subset of $X(\bk)$ consisting of $f$-periodic points is Zariski dense in $X$.
Pick an $f$-periodic point $y\in X(\bk)\setminus Z_f(x)$.
After iterating $f$, we may assume that $y$ is an $f$-fixed point.
Take a rational curve $C\subseteq X$ containing $x$ and $y$.
Set $W\coloneqq Z_f(C)$.
If $W = X$, we are done.
So we may assume that $W\subsetneq X$.
If $\dim W = r$, then $W$ has its irreducible decomposition as $W = Z_f(x)\cup W_1\cup\cdots\cup W_m$.
There is some $n\geq 0$ such that $f^n(x)\in Z_f(x)\setminus\bigcup_{i=1}^m W_i$.
Then $f^n(C)\subseteq W$ but $f^n(C)\not\subseteq\bigcup_{i=1}^m W_i$.
Hence $f^n(C)\subseteq Z_f(x)$.
In particular, $y = f^n(y) \in f^n(C)\subseteq Z_f(x)$, a contradiction.
Thus $r < \dim W$ ($< \dim X = d$).

Now there exists an $f$-periodic irreducible component $W' \subseteq W$ with $r < \dim W' < d$.
Replacing $f$ by a positive power, we may assume that $W'$ is $f$-invariant.
Then $f|_{W'}$ is an int-amplified endomorphism on $W'$ (cf.~\cite{meng2020building}*{Lemma 2.2}).
By assumption, Conjecture~\ref{Conj:ZDO} holds for $(W', f|_{W'})$.
So there exists some $w\in W'(\bk)$ such that $Z_f(w) = Z_{f|_{W'}}(w) = W'$ (cf.~Lemma~\ref{l:int-amp_ZDO}).
In particular, $Z_f(w)$ is irreducible with $\dim Z_f(w) > r$.
Continuing this process, the lemma follows.
\end{proof}

\begin{corollary}\label{c:pd_rc}
Let $X$ be a rationally connected projective variety over $\bk$ and of dimension $\leq 3$, and $f\colon X\to X$ an int-amplified endomorphism.
Then there is a rational curve $C \subseteq X$ such that $O_f(C)$ is Zariski dense.
\end{corollary}

\begin{proof}
This follows from Remark~\ref{r:after_conj}\ref{r:ZDO_curve}, \ref{r:ZDO_surf}, and Lemma~\ref{l:pd_rc}.
\end{proof}

\begin{proof}[Proof of Proposition~\ref{t:pd_rc}]
By applying Corollary~\ref{c:pd_rc} to $(X_{\oK},f_{\oK})$, we know that there is a rational curve $C\subseteq X_{\oK}$ such that $O_{f_{\oK}}(C)$ is Zariski dense.
Replacing $K$ with a finite extension, we may assume that $C$ is defined over $K$.
Then $O_f(C)$ is Zariski dense in $X$.
The theorem follows from Lemma~\ref{l:subvar}.
\end{proof}

\section{Int-amplified endomorphisms: Proofs of Proposition~\ref{t:non-ur} and Theorem~\ref{t:main}}
\label{s:main}

\begin{lemma}(cf.~\cite{meng2020building}*{Theorem 1.9})
\label{l:non-ur}
Let $X$ be a normal projective variety over $\bk$ and $f \colon X \to X$ an int-amplified endomorphism.
Assume one of the following conditions.
\begin{enumerate}
	\item $X$ is non-uniruled.
	\item $X$ has at worst $\bQ$-factorial klt singularities, and $K_X$ is pseudo-effective.
\end{enumerate}
Then $X$ is a \textit{Q}-abelian variety.
In particular, $f$ has a Zariski dense orbit.
\end{lemma}

\begin{proof}
The first claim is \cite{meng2020building}*{Theorem 1.9}.
Now there is a finite cover $\pi\colon A\to X$ ({\'e}tale in codimension $1$) from an abelian variety $A$ with $f$ lifted to an int-amplified endomorphism $g$ on $A$ (cf.~\cite{nakayama2010polarized}*{Lemma 2.12} and \cite{meng2020building}*{Lemma 3.5}).
Since Conjecture~\ref{Conj:ZDO} holds for endomorphisms on abelian varieties (cf.~\cite{ghioca2017density}), $g$ has a Zariski dense orbit $O_g(a)$ for some $a\in A(\bk)$ (cf.~Lemma~\ref{l:int-amp_ZDO}).
Then $O_f(\pi(a))$ is a Zariski dense orbit of $f$.
\end{proof}

\begin{lemma}\label{l:main}
Let $X$ be a normal projective variety over $\bk$ and of dimension $\leq 3$ with at worst $\bQ$-factorial klt singularities.
Let $f \colon X \to X$ be an int-amplified endomorphism.
Then there exists a rational subvariety $Z \subseteq X$ of dimension $\geq 0$, such that $O_f(Z)$ is Zariski dense.
\end{lemma}

\begin{proof}
By Remark~\ref{r:after_conj}\ref{r:ZDO_curve}, \ref{r:ZDO_surf}, and Lemma~\ref{l:int-amp_ZDO}, the assertion holds when $\dim X\leq 2$.
Then by Corollary~\ref{c:pd_rc} and Lemma~\ref{l:non-ur}, we may assume that $X$ is a threefold, which is uniruled but not rationally connected, and $K_X$ is not pseudo-effective.

By \cite{meng2020semi}, replacing $f$ with an iteration, we can run an $f$-equivariant minimal model program:
\[
\xymatrix{
	X=X_0 \ar@{-->}[r]^-{\mu_0} &X_1 \ar@{-->}[r]^-{\mu_1} &\cdots \ar@{-->}[r]^-{\mu_{m-1}} &X_m=X' \ar[r]^-{\pi} &Y,
}
\]
where each $\mu_i$ is a birational map and $\pi$ is a Mori fibre space with $\dim Y < \dim X' = 3$.
If $\dim Y = 0$, then $X'$ is klt Fano.
Hence $X'$ and $X$ are rationally connected (cf.~\cite{zhang2006rational}*{Theorem 1}), contradicting our extra assumption.
Thus $\dim Y = 1, 2$.
Since $\dim Y \leq 2$, the int-amplified endomorphism $g\coloneqq f|_Y$ has a Zariski dense orbit $O_g(y)$ by Remark~\ref{r:after_conj} \ref{r:ZDO_curve}, \ref{r:ZDO_surf} and Lemma~\ref{l:int-amp_ZDO} (cf.~\cite{meng2020building}*{Lemmas 3.4 and 3.5}).
Replacing $y$ by $g^N(y)$ for a suitable $N \geq 0$, we may assume that $F\coloneqq \pi^{-1}(y)$ is a klt Fano variety of dimension equal to $\dim X - \dim Y \in\{1, 2\}$, and hence a rational variety.
Clearly, $O_f(F)$ is Zariski dense in $X$ by construction.
\end{proof}

\begin{proof}[Proof of Proposition~\ref{t:non-ur}]
Since being uniruled and the potential density are birational properties (cf.~\ref{notation}\ref{n:ur} and \cite{hassett2003potential}*{Proposition 3.1}), they are invariant under the normalisation map.
Also, since an int-amplified endomorphism on the variety $X$ lifts to an int-amplified endomorphism on its normalisation (cf.~\cite{meng2020building}*{Lemma 3.5}), we may assume that $X$ is normal.
Then the proposition follows from Lemmas~\ref{l:non-ur} and \ref{l:subvar}.
\end{proof}

\begin{proof}[Proof of Theorem~\ref{t:main}]
This follows from Lemmas~\ref{l:main} and \ref{l:subvar}.
\end{proof}

\section{The maximal arithmetic degree: Proofs of Theorems~\ref{t:surf} and \ref{t:rc_dr}}
\label{s:dr}

In this section, we study Question~\ref{q:dr}.
First, we prove Theorem~\ref{t:surf}.

\begin{lemma}\label{l:bir_stable}
Let $X, Y$ be normal projective varieties over $K$, and $f\colon X\to X$ and $g\colon Y\to Y$ surjective endomorphisms.
Assume that there is a surjective morphism $\pi\colon X\to Y$ such that $\pi\circ f = g\circ\pi$. Then:
\begin{enumerate}
	\item If $\pi$ is generically finite and $(X, f)$ satisfies $(DR)$, then $(Y, g)$ also satisfies $(DR)$.
	\item If $\pi$ is birational, then $(X, f)$ satisfies $(DR)$ if and only if so does $(Y, g)$.
\end{enumerate}
\end{lemma}

\begin{proof}
Assume first that $\pi$ is generically finite.
Let $X \xrightarrow{\pi'} X' \xrightarrow{\varphi} Y$ be the Stein factorisation of $\pi$, where $\pi'$ is a projective morphism with connected fibres (indeed, $\pi'_*\cO_X\simeq\cO_{X'}$) to a normal variety $X'$, and $\varphi$ is a finite morphism (cf.~\cite{hartshorne1977algebraic}*{Corollary 11.5}).
Since $\pi\circ f = g\circ\pi$ and $\varphi$ is finite, we see that $\pi'\circ f$ contracts every fibre of $\pi'$.
By the rigidity lemma (cf.~\cite{debarre2001higher}*{Lemma 1.15}), there is a morphism $f'\colon X'\to X'$ such that $\pi'\circ f = f'\circ\pi'$ and $\varphi\circ f' = g\circ\pi$.
By \cite{sano2021zariski}*{Lemma 3.2}, for (1), we only need to show that $(X', f')$ satisfies $(DR)$, which can be deduced from (2);
for (2), we only need to show that if $(X, f)$ satisfies $(DR)$, then so does $(Y, g)$.

Let $\Sigma\subseteq Y$ be the subset consisting of points $y$ such that $\dim\pi^{-1}(y) > 0$, and $E\coloneqq\pi^{-1}(\Sigma)\subseteq X$, which is a closed proper subset.
Since $\pi$ has connected fibres by Zariski's Main Theorem (cf.~\cite{hartshorne1977algebraic}*{Corollary 11.4}), $\pi|_{X\setminus E}\colon X\setminus E\to Y\setminus\Sigma$ is an isomorphism.
Since $g$ is finite, both $\Sigma$ and $Y\setminus \Sigma$ are $g^{-1}$-invariant.
There is an induced surjective morphism $f|_{X\setminus E}\colon X\setminus E\to X\setminus E$ such that $\pi|_{X\setminus E}\circ f|_{X\setminus E}=g|_{Y\setminus\Sigma}\circ\pi|_{X\setminus E}$.
Let $L$ be a finite field extension of $K$ such that $(X,f)$ satisfies $(DR)_L$.
Then there exists a sequence of $L$-rational points $S_X=\{x_i\}_{i=1}^{\infty}\subseteq X(L)\setminus E$ such that
\begin{itemize}
	\item $S_X$ is Zariski dense in $X_L$;
	\item $\alpha_f(x_i)=d_1(f)$ for all $i$; and
	\item $O_f(x_i)\cap O_f(x_j)=\emptyset$ for $i\neq j$.
\end{itemize}
Thus $y_i\coloneqq \pi(x_i)$ is well-defined and $S_Y\coloneqq \{y_i\}_{i=1}^{\infty}$ satisfies the conditions of $(DR)_L$ for $(Y, g)$;
note that $d_1(f)=d_1(g)$ and $\alpha_f(x_i)=\alpha_g(y_i)$ (cf.~\cite{silverman2017arithmetic}*{Lemma 3.2} in the smooth case, or \cite{matsuzawa2020nondensity}*{Lemma 2.8} in general).
\end{proof}

We need the following from \cite{sano2020zariski}.

\begin{lemma}[cf.~{\cite{sano2020zariski}*{Theorem 4.1}}]
\label{l:dagger}
Let $X$ be a projective variety over $K$ and $f \colon X \to X$ a surjective morphism with $d_1(f)>1$.
Assume the following condition:
\begin{align}\label{eq:property}
	\tag{$\dagger$}
	&\text{There is a numerically non-zero nef $\bR$-Cartier divisor $D$ on $X$ such that}\\[-0.3em]
	&\text{$f^*D \sim_{\bR} d_1(f)D$, and for any proper closed subset $Y \subseteq X_{\oK}$, there exists}\notag\\[-0.3em]
	&\text{a morphism $g\colon \bP^1_K \to X$ such that $g(\bP^1_K) \not\subseteq Y$ and $g^*D$ is ample.}\notag
\end{align}
Then $(X, f)$ satisfies $(DR)_K$.
\end{lemma}

We also need the following structure theorem of endomorphisms.

\begin{proposition}[cf.~\cite{jia2020surjective}*{Theorem 1.1}]
\label{p:surf_cls}
Let $f\colon X\to X$ be a non-isomorphic surjective endomorphism of a normal projective surface over $\bk$.
Then, replacing f with a positive power, one of the following holds.
\begin{enumerate}[label=(\roman*)]
	\item $\rho(X)=2$;
	there is a $\bP^1$-fibration $X\to C$ to a smooth projective curve of genus $\geq 1$, and $f$ descends to an automorphism of finite order on the curve $C$.
	\item $f$ lifts to an endomorphism $f|_V$ on a smooth projective surface $V$ via a generically finite surjective morphism $V \to X$.
	\item $X$ is a rational surface.
\end{enumerate}
\end{proposition}

\begin{proof}
We use \cite{jia2020surjective}*{Theorem 1.1}.
Cases (1), (3) and (8) imply our (ii).
Cases (4) $\sim$ (7) and (9) lead to our (iii).
Case (2) implies our (i), noting that $f$ cannot be polarised since it descends to an automorphism and hence $\rho(X) = 2$ by \cite{meng2019kawaguchi}*{Theorem 5.4}.
\end{proof}

\begin{proof}[Proof of Theorem {\ref{t:surf}}]
When $f$ is an automorphism, we may take an equivariant resolution of $(X, f)$ and assume that $X$ is smooth (cf.~Lemma~\ref{l:bir_stable}).
In this case, the theorem follows from \cite{sano2021zariski}*{Theorem 1.5}.

Now we assume that $\deg(f)\geq 2$.
We apply Proposition~\ref{p:surf_cls} to $(X_{\oK},f_{\oK})$ (cf.~\cite{sano2021zariski}*{Lemma 3.3}).
In either case, we may replace $K$ with a finite field extension so that the varieties and morphisms are defined over $K$.

In Case~\ref{p:surf_cls} (ii), the theorem follows from Lemma~\ref{l:bir_stable} and \cite{sano2021zariski}*{Theorem 1.5}.
In Case~\ref{p:surf_cls} (iii), the theorem is a consequence of \cite{sano2020zariski}*{Theorem 1.11}.

In Case~\ref{p:surf_cls} (i), we may assume $g(C) = 1$;
otherwise, $X$ does not satisfy potential density (cf.~Remark~\ref{r:after_conj}\ref{r:nonPD}).
Let $F \cong \bP^1$ be a general fibre of $X \to C$.
After replacing $K$ with a finite field extension, there is a numerically non-zero nef $\bR$-Cartier divisor $D$ on $X$ such that $f^*D \sim_{\bR} d_1(f) D$ (cf.~\cite{matsuzawa2021invariant}*{Theorem 6.4}).
The numerical equivalence class of $D$ is not a multiple of that of the fibre $F$ since $f^*F\sim_{\bR}F$ and $d_1(f)>1$.
Then $(D \cdot F)>0$, by the Hodge index theorem.
Thus, $(X, f)$ satisfies the condition \eqref{eq:property} in Lemma~\ref{l:dagger} and hence satisfies $(DR)$.
\end{proof}

Before proving Theorem~\ref{t:rc_dr}, we need a stronger version of Corollary~\ref{c:pd_rc} in dimension $3$.

\begin{lemma}\label{l:pd_rc2}
Let $X$ be a rationally connected smooth projective threefold over $\bk$ and $f\colon X\to X$ an int-amplified endomorphism.
Let $D$ be a numerically non-zero nef $\bR$-Cartier divisor on $X$.
Then there is a rational curve $C\subseteq X$ such that $O_f(C)$ is Zariski dense and $(D \cdot C)>0$.
\end{lemma}

\begin{proof}
By \cite{yoshikawa2020structure}*{Corollary 1.4}, $X$ is of Fano type.
Then there is a surjective morphism $\phi\colon X \to Y$ to a projective variety $Y$ such that $D \sim_{\bR} \phi^*H$ for some ample $\bR$-divisor on $Y$ by \cite{birkar2010existence}*{Theorem 3.9.1}.

If $f$ has a Zariski dense orbit $O_f(x)$, then there is a rational curve passing through $x$ (such a curve exists since $X$ is rationally connected) and satisfying the claims.
So we may assume that $f$ has no Zariski dense orbit.

Since Conjecture~\ref{Conj:ZDO} is known for surfaces (cf.~\cite{jia2020surjective}*{Theorem 1.9}), we can take a point $x_0 \in X$ such that $\dim Z_f(x_0)=2$ (cf.~Proof of Lemma~\ref{l:pd_rc}).
Replacing $f$ by a power and $x_0$ by $f^N(x_0)$ for some integer $N\geq 0$, we may assume that $Z_f(x_0)$ is irreducible.
Take an $f$-periodic point $x_1 \in X$ such that $x_1 \not\in Z_f(x_0) \cup \phi^{-1}(\phi(x_0))$.
Take a rational curve $C \subseteq X$ containing $x_0, x_1$.
We see that $O_f(C)$ is Zariski dense as in the proof of Lemma~\ref{l:pd_rc}.
Now $\phi(C)$ is not a point by construction, so
\[
	(D \cdot C)=(\phi^*H \cdot C)=(H \cdot \phi_*C)>0.
\]
Thus $C$ satisfies the claims.
\end{proof}

\begin{proof}[Proof of Theorem~\ref{t:rc_dr}]
By \cite{matsuzawa2021invariant}*{Theorem 6.4}, replacing $K$ by a finite extension, there is a numerically non-zero nef $\bR$-Cartier divisor $D$ on $X$ such that $f^*D \sim_{\bR} d_1(f) D$.
Lemma~\ref{l:pd_rc2} implies that, replacing $K$ with a finite extension so that the curve $C$ there (and $f$) are defined over $K$, the pair $(X,f)$ satisfies \eqref{eq:property} in Lemma~\ref{l:dagger}.
Hence $(X,f)$ satisfies $(DR)$.
\end{proof}

\end{document}